\documentclass[10pt]{article}

\usepackage{latexsym}
\usepackage{amsbsy}
\usepackage{amssymb}
\usepackage{amsfonts}
\usepackage{amsmath,amsthm}
\usepackage[latin1]{inputenc}%needed for bibliography
\usepackage{a4wide}

\usepackage[color]{changebar}
\cbcolor{blue}

 \newtheorem{thm}{Theorem}[section]
 \newtheorem{defin}[thm]{Definition}
 \newtheorem{lem}[thm]{Lemma}
 \newtheorem{prop}[thm]{Proposition}
 \newtheorem{cor}[thm]{Corollary}
 \theoremstyle{definition} % added by gue to make text in rem and ex non-italic
 \newtheorem{rem}[thm]{Remark}
 \newtheorem{ex}[thm]{Example}

 \newcommand{\bthm}{\begin{thm}}
 \newcommand{\ethm}{\end{thm}}
 \newcommand{\bd}{\begin{defin}}
 \newcommand{\ed}{\end{defin}}
 \newcommand{\blem}{\begin{lem}}
 \newcommand{\elem}{\end{lem}}
 \newcommand{\bcor}{\begin{cor}}
 \newcommand{\ecor}{\end{cor}}
 \newcommand{\bprop}{\begin{prop}}
 \newcommand{\eprop}{\end{prop}}
 \newcommand{\brem}{\begin{rem} \rm}
 \newcommand{\erem}{\end{rem}}
 \newcommand{\bex}{\begin{ex} \rm}
 \newcommand{\eex}{\end{ex}}

 \newcommand{\pr}{\noindent{\bf Proof. }}
 \newcommand{\ep}{\nolinebreak{\hspace*{\fill}$\Box$ \vspace*{0.25cm}}}

 \newcommand{\beq}{\begin{equation}}
 \newcommand{\eeq}{\end{equation} }

 \newcommand{\bea}{\begin{eqnarray}}
 \newcommand{\eea}{\end{eqnarray}}

 \newcommand{\beas}{\begin{eqnarray*}}
 \newcommand{\eeas}{\end{eqnarray*}}

 \newcommand{\beqs}{\begin{equation*}}
 \newcommand{\eeqs}{\end{equation*}}

 \newcommand{\bi}{\begin{itemize}}
 \newcommand{\ei}{\end{itemize}}

 \newcommand{\ben}{\begin{enumerate}}
 \newcommand{\een}{\end{enumerate}}

 \newcommand{\ba}{\begin{array}}
 \newcommand{\ea}{\end{array}}

 \newcommand{\R}{\mathbb R}
 \newcommand{\N}{\mathbb N}

 \newcommand{\Rt}{\widetilde{\mathbb R}}

 \newcommand{\sig}{\sigma}
 \newcommand{\de}{\delta}
 \newcommand{\la}{\lambda}
 \newcommand{\al}{\alpha}
 \newcommand{\be}{\beta}
 \newcommand{\om}{\omega}
 \newcommand{\col}{\colon}
 \newcommand{\dual}[2]{\langle #1, #2 \rangle}
 
 \renewcommand{\labelenumi}{(\roman{enumi})}

 \newcommand{\diff}[1]{\frac{d}{d#1}}

 \newcommand{\cA}{\ensuremath{{\mathcal A}}}
 
 \newcommand{\cC}{\ensuremath{{\mathcal C}}}
 
 \newcommand{\cE}{\ensuremath{{\mathcal E}}}
 
 \newcommand{\cG}{\ensuremath{{\mathcal G}}}

 \newcommand{\cN}{\ensuremath{{\mathcal N}}}
 
 \newcommand{\cP}{\ensuremath{{\mathcal P}}}

 \newcommand{\eps}{\varepsilon}
 \newcommand{\vphi}{\varphi}

 \newcommand{\Vol}{\ensuremath{\mbox{Vol\,}}}

\newcommand{\norm}[1]{\| #1\|}
\newcommand{\opnorm}[1]{\| #1\|_{\text{op}}}

\newcommand{\vol}{\mathrm{vol}}
\begin{document}

 \title{A regularization approach to non-smooth symplectic geometry
}

 \author{G\"unther H\"ormann
         \footnote{Faculty of Mathematics, University of Vienna,
         Oskar-Morgenstern-Platz 1, A-1090 Vienna, Austria,
         Electronic mail: guenther.hoermann@univie.ac.at}\\
         Sanja Konjik
         \footnote{Faculty of Sciences, Department of Mathematics and Informatics, University of Novi Sad,
         Trg Dositeja Obradovi\'ca 4, 21000 Novi Sad, Serbia,
         Electronic mail: sanja.konjik@dmi.uns.ac.rs}\\
         Michael Kunzinger
         \footnote{Faculty of Mathematics, University of Vienna,
         Oskar-Morgenstern-Platz 1, A-1090 Vienna, Austria,
         Electronic mail: michael.kunzinger@univie.ac.at.}
       }

 \date{}
 \maketitle

 \begin{abstract}
We introduce non-smooth symplectic forms on manifolds and describe
corresponding Poisson structures on the algebra of Colombeau generalized
functions. This is achieved by establishing an extension of the classical
map of smooth functions to Hamiltonian vector fields to the setting of
non-smooth geometry. For mildly singular symplectic
forms, including the continuous non-differentiable case, we prove the
existence of generalized Darboux coordinates in the sense of a local
non-smooth pull-back to the canonical symplectic form on the cotangent
bundle.
 \end{abstract}

%%%%%%%%%%%%%%%%%%%%%%%%%%%%%%%%%%%%%%%%%%%%%%%%%%%%%%%%%%%%%%%%%%
 \section{Introduction} \label{sec:intro}
%%%%%%%%%%%%%%%%%%%%%%%%%%%%%%%%%%%%%%%%%%%%%%%%%%%%%%%%%%%%%%%%%%

Regularization approaches to non-smooth differential geometry and its
applications to mathematical physics have been successfully developed in
the context of Colombeau-type generalized functions and tensor fields 
(cf., e.g., \cite{C:84,C85,NPS98,O:92,book}). In the present paper we take
up the study of generalized symplectic structures based on previous
investigations  of linear symplectic structures on modules over Colombeau
generalized numbers in \cite{HKK:14}. Our main motivations for the
systematic development of a non-smooth symplectic differential geometry
are driven by deeper applications in microlocal analysis, classical
mechanics or general relativity in terms of analysis on semi-Riemannian
manifolds:

First, modern research on propagation of singularities for
(pseudo-) differential
operators with non-smooth principal symbol on a manifold %demands to grasp
is based on an understanding of
the corresponding non-smooth Hamiltonian vector field and its generalized
bicharacteristic flow  (cf.\ \cite{GH:05,O:09}) as well as on an analysis of 
microlocal mapping properties of generalized Fourier integral solution
operators in terms of the wave front sets of their kernels, which are to
be  generalized Lagrangian submanifolds  (cf.\
\cite{GHO:09,GO:14}).

Second, non-smooth symplectic structures arise in the study of the
geodesic flow in classical mechanics or in general relativity in the
context of non-smooth metrics or space-times, unified in models of 
generalized semi-Riemannian manifolds $(M,g)$. The geodesic flow can then
be described in terms of symplectomorphisms on a generalized symplectic
manifold $(TM, \sig)$. The basic construction is as follows: The geodesic
flow is defined by the non-smooth geodesic spray $G$, given as a vector
field on $TM$  in coordinates $(x,v)$ by
$$
   G(x,v) =  \sum_{1 \leq j \leq n}  v_j\, \partial_{x_j} -
   \sum_{1 \leq j,k,l \leq n} \Gamma^j_{k l}\, v_k\, v_l\, \partial_{v_j}
$$
with  the Christoffel symbols $\Gamma^j_{k l}$ (cf.\
\cite{AM:78}). Non-degeneracy of the metric provides a `non-smooth
diffeomorphism' $g^\flat \col TM \to T^* M$. The latter allows one to
pull-back the canonical symplectic form to define a non-smooth symplectic
form $\sig$  on $TM$, which locally reads
$$
  \sig = \sum_{1 \leq i,j \leq n} g_{ij}\, dx_i \wedge dv_j +
  \sum_{i,j,k} \frac{\partial g_{ij}}{\partial{x_k}} \, v_i\, dx_j \wedge
dx_k.
$$
Note that in a sense, here the `generalized symplectomorphism'
$(g^\flat)^{-1}$ provides Darboux coordinates for $(TM, \sig)$.

\subsection{Generalized differential geometry on smooth manifolds}

In this section we briefly recall some notions from Colombeau's theory
of nonlinear generalized functions and non-smooth differential geometry 
in this setting. For details we refer to \cite{book}.

Let $M$ be a smooth (Hausdorff and second countable) manifold of dimension $n$.
Colombeau generalized functions on $M$ are introduced as 
equivalence classes $u=[(u_\eps)_\eps]$ of moderate modulo negligible
nets in $\cC^\infty(M)$, where moderateness, resp. negligibility, are 
characterized by
\beas
{\mathcal E}_M(M)&:=&\{ (u_\eps)_\eps\in{\mathcal C}^\infty(M)^{]0,1]}:\ \forall
K\subset\subset M,\ \forall L\in \cP(M)\ \exists N\in\N:
 \sup_{x\in K}|Lu_\eps(x)|=O(\eps^{-N})\}, \\
{\mathcal N}(M)&:=& \{ (u_\eps)_\eps\in\cE_M(M):\ \forall K\subset\subset M,\ 
\forall m \in\N_0:\ \sup_{x\in K}|u_\eps(x)|=O(\eps^{m})\}, 
\eeas
and $\cP(M)$ is the space of linear differential operators on $M$. 
Then the Colombeau algebra $\cG(M)$ of generalized functions on $M$ 
is defined as is the quotient $\cE_M(M)/\cN(M)$; it is 
a fine sheaf of differential algebras with respect to the
Lie derivative along smooth vector fields.

Colombeau generalized functions on $M$ are uniquely determined by 
their values on compactly supported generalized points on $M$, which 
are denoted by $\widetilde{M}_c$ and defined as follows.
In the space $M_c$ of nets $(x_\eps)_\eps\in M^{]0,1]}$ with the
property that $x_\eps$ stays in a fixed compact set for $\eps$ small,
one introduces an equivalence relation $\sim$: 
$(x_\eps)_\eps\sim (y_\eps)_\eps :\Leftrightarrow 
d_h(x_\eps,y_\eps)=O(\eps^m)$,
for all $m>0$, $(x_\eps)_\eps,\, (y_\eps)_\eps\in M_c$, 
and distance function $d_h$ induced on
$M$ by any Riemannian metric $h$. Then the space of compactly
supported generalized points on $M$ is the quotient space
$\widetilde{M}_c:=M_c/\sim$, with elements $\tilde{x}=[(x_\eps)_\eps]$.

Generalized numbers
are equivalence classes $r=[(r_\eps)_\eps]$ of moderate
nets of real numbers 
$\{(r_\eps)_\eps \in \R^{]0,1]} : \exists N\in\N: |r_\eps| = O(\eps^{-N})\}$ 
modulo the set of negligible nets 
$\{(r_\eps)_\eps \in \R^{]0,1]} : \forall m\in\N_0: |r_\eps| = O(\eps^{m})\}$.
$\widetilde{\R}$ is the ring of constants in the Colombeau algebra of generalized
functions on $\R$. A generalized number $r\in\widetilde{\R}$ is called strictly
nonzero if $|r_\eps|\geq \eps^m$ for some $m\in\N$ and $\eps$ small; it is 
called strictly positive if $r_\eps\geq \eps^m$ for some $m\in\N$ and $\eps$ small.
Invertible elements of the ring $\widetilde{\R}$ are precisely those which
are strictly nonzero. Similarly, $u\in \cG(M)$ has a multiplicative inverse if and
only if it is strictly nonzero in the sense that for any $K\Subset M$ there exists
some $m$ such that $\inf_{x\in K}|u_\eps(x)|\ge \eps^m$ for $\eps$ small. In turn,
this is equivalent to $u(\tilde x)$ being invertible in $\Rt$ for any compactly
supported generalized point $\tilde x$.
 
Let $E$ be a vector bundle over $M$, and denote by $\Gamma(E)$ the
space of smooth sections of $E$.
The space of Colombeau generalized sections of $E$, $\Gamma_\cG(E)$, 
is defined as the quotient $\Gamma_{\cE_M}(E)/\Gamma_{\cN}(E)$, where
\beas
\Gamma_{\cE_M}(E)&:=& \{ (s_\eps)_{\eps}\in \Gamma(E)^{]0,1]} :
      \ \forall L\in \cP(E)\, \forall K\subset\subset M \, \exists N\in \N:
\sup_{x\in K}\|Lu_\eps(x)\|_h = O(\eps^{-N})\},\\
\Gamma_{\cN}(E)&:=& \{ (s_\eps)_{\eps}\in \Gamma_{\cE_M}(E) :
      \ \forall K\subset\subset M \, \forall m\in \N:
\sup_{x\in K}\|u_\eps(x)\|_h = O(\eps^{m})\},
\eeas
Here $\cP(E)$ is the space of linear differential operators 
$\Gamma(E)\to\Gamma(E)$, and $\|\,\|_h$ is the norm on 
the fibers of $E$ induced by any Riemannian metric $h$ on $M$.
The ${\mathcal C}^\infty(M)$-module of Colombeau generalized sections of $E$ 
is projective and finitely generated, and 
can be characterized
by the $\cC^\infty(M)$-module isomorphisms:
$\Gamma_{\cG}(E)=\cG(M)\otimes_{\cC^\infty(M)}\Gamma(E)
= L_{\cC^\infty(M)}(\Gamma(E^\ast),\cG(M))$.

In case $E$ is a tensor bundle $T_s^r(M)$ we use the notation
$\cG_s^r(M)$ for $\Gamma_{\cG}(T_s^r(M))$. Moreover, in the case of 
the tangent bundle $TM$ the generalized sections are the
generalized vector fields, and will be denoted by $\mathfrak{X}_\cG(M)$,
while in the case of the cotangent bundle $T^\ast M$ we write 
$\Omega_\cG^1(M)$ for the corresponding generalized sections (generalized one-forms). 
Also, when $E$ is the vector bundle of exterior 
$k$-forms on $TM$, i.e., $E=\Lambda^kT^\ast M$, then the generalized
sections are the generalized $k$-forms on $M$, and are denoted by 
$\Omega_\cG^k(M)$.

Finally, we will also make use of a particular feature of Colombeau's approach, 
namely manifold-valued generalized functions (\cite{K:02,KSV:03}). 
Given manifolds $M$, $N$, the space of c-bounded generalized functions
is the quotient space $\cG[M,N]:=\cE_M[M,N]/\sim$. Here, $\cE_M[M,N]$ is the
set of all nets $(u_\eps)_{\eps\in ]0,1]}$ such that $(f\circ u_\eps)_\eps$
is moderate for every $f\in {\mathcal C}^\infty(N)$. The equivalence relation
$\sim$ is defined as $(u_\eps)_\eps \sim (v_\eps)_\eps$ if for any Riemannian
metric $h$ on $N$, any $m\in \N$ and every $K\Subset M$, 
$\sup_{p\in K}d_h(u_\eps(p),v_\eps(p))=O(\eps^m)$ for $\eps\to 0$.
These generalized functions are called c-bounded since any representative
is bounded, uniformly in $\eps$, on any compact subset of $M$ for $\eps$ small.
They can be composed unrestrictedly, and invertible c-bounded generalized functions
serve as non-smooth analogues of diffeomorphisms in the smooth category
(cf.\ \cite{EG:11,V:11}).

The paper is organized as follows. For the sake of completeness
we first review in Section 2 some results from \cite{HKK:14} 
on symplectic forms on $\Rt$-modules, symplectic bases, 
maps, and submodules. Then we turn to the manifold setting, study various conditions
for skew-symmetry and nondegeneracy of generalized $2$-forms, 
and provide an equivalent characterization of a generalized symplectic 
form on a manifold. Section 3 is devoted to the Darboux theorem
and its generalization for a generalized symplectic form. We
state conditions which imply that a generalized symplectic form
on a manifold looks locally like the canonical symplectic form on 
$T^\ast(\widetilde{\R}^n)$.
In the last Section 4 we introduce notions of generalized Hamiltonian
vector fields and Poisson structures.

%%%%%%%%%%%%%%%%%%%%%%%%%%%%%%%%%%%%%%%%%%%%%%%%%%%%%%%%%%%%%%%%%%
 \section{Generalized symplectic structures} \label{sec:gf}
%%%%%%%%%%%%%%%%%%%%%%%%%%%%%%%%%%%%%%%%%%%%%%%%%%%%%%%%%%%%%%%%%%

%%%%%%%%%%%%%%%%%%%%%%%%%%%%%%%%%%%%%%%%%%%%%%%%%%%%%%%%%%%%%%%%%%
\subsection{Review of symplectic modules over the ring of generalized numbers}\label{subsec:symplmod}
%%%%%%%%%%%%%%%%%%%%%%%%%%%%%%%%%%%%%%%%%%%%%%%%%%%%%%%%%%%%%%%%%%

Here we briefly recall basic notions and results about symplectic $\Rt$-modules
 that are essential for the study of generalized symplectic structures
 on manifolds, and refer to \cite{HKK:14} for an in-depth analysis.
 
 On an $\Rt$-module $V$ we define a symplectic form $\sig:V \times V \to \Rt$
 to be an $\Rt$-bilinear form that is 
 skew-symmetric ($\sig(v,w) = - \sig(w,v), \forall v, w \in V$), 
 and non-degenerate ($\sig(v,w) = 0, \forall w \implies v = 0$).
 We call the pair $(V,\sig)$ a symplectic $\Rt$-module.
 
 As the standard model space for a symplectic $\Rt$-module we take 
 $(T^*({\Rt}^n), \tilde{\omega})$, where $T^*({\Rt}^n)= \Rt^n \times \Rt^n$, 
 and the symplectic form $\tilde{\omega}$ is defined as 
\beq\label{standsymp}
  \widetilde{\omega}((x,\xi),(y,\eta)) = \sum_{j=1}^n y_j \xi_j - 
  \sum_{j=1}^n x_j \eta_j
  = \dual{y}{\xi} - \dual{x}{\eta}
  \qquad \forall (x,\xi), (y,\eta) \in T^*(\Rt^n).
\eeq
 For this symplectic form, the vectors $\{e_1,\ldots,e_n,f_1,\ldots,f_n\}$, $e_j := (\de_j,0)$, $f_j = (0,\de_j)$ 
 ($\de_j$ is the $j^\text{th}$ standard unit vector, $1 \leq j \leq n$),
 form a basis which turns $T^*(\Rt^n)$ into a free module of rank $2n$.
 Moreover, one has
 $$
  \widetilde{\omega}(e_j,e_l) = 0 = \widetilde{\omega}(f_j,f_l), 
  \quad \widetilde{\omega}(f_j,e_l) = \delta_{jl}
  \quad (1 \leq j,l \leq n),
 $$
 with $\delta_{jl}$  the Kronecker delta. A basis satisfying this
 property is called a symplectic basis.
 
 In \cite{HKK:14} we proved that any symplectic free module of finite rank 
 possesses a symplectic basis, which further implies that its rank
 has to be even. Also, we showed that any given ``partial symplectic basis''
 of the free symplectic $\Rt$-module $(V,\sig)$ of finite rank can be extended to a full one,
 i.e., any free set $B:= \{e_i\in V \mid i \in I \} \cup \{f_j\in V \mid j \in J\}$
 ($I, J \subseteq \{ 1,\ldots n\}$) satisfying
$$
  \sig(e_i,e_k) = 0 = \sig(f_j,f_l), \quad \sig(f_j,e_i) = \delta_{ji}
  \quad (i,k \in I; j,l \in J),
$$
 can be extended by vectors $e_i \in V$ ($i \in \{1, \ldots n\} \setminus I$) and
 $f_j \in V$ ($j \in \{1, \ldots n\} \setminus J$) to a symplectic basis 
 $\{e_1, \ldots, e_n, f_1, \ldots, f_n\}$ of $(V,\sig)$.
 
 One of the main results of \cite{HKK:14} is that any free symplectic
 $\Rt$-module of finite rank is symplectomorphic to $(T^*(\Rt^n),\widetilde{\omega})$.
 A symplectomorphism between symplectic $\Rt$-modules $(V_1,\sig_1)$ and $(V_2,\sig_2)$
 is an $\Rt$-linear isomorphism $f:V_1\to V_2$ that preserves symplectic structures, i.e.,
 $$
   \sig_2(f(v_1),f(v_2)) = \sig_1(v_1,v_2), \quad \forall v_1, v_2 \in V. 
 $$
 In case $f$ is an $\Rt$-linear map but not an isomorphism, it
 is called a symplectic map. Every symplectic map is injective.
 A symplectic map $f:V_1\to V_2$ is a symplectomorphism if, in
 addition, $(V_1,\sig_1)$ and $(V_2,\sig_2)$ are free and of equal finite rank.

%%%%%%%%%%%%%%%%%%%%%%%%%%%%%%%%%%%%%%%%%%%%%%%%%%%%%%%%%%%%%%%%%%
\subsection{Manifolds with generalized symplectic forms}\label{subsec:symplmanifolds}
%%%%%%%%%%%%%%%%%%%%%%%%%%%%%%%%%%%%%%%%%%%%%%%%%%%%%%%%%%%%%%%%%%
 
The following definition is the natural extension of the notion of a smooth
symplectic form on a manifold to the setting of Colombeau generalized functions.
 
\bd
A \emph{generalized symplectic form} on the smooth $d$-dimensional manifold $M$ is a closed generalized $2$-form $\sig \in \Gamma_\cG(\Lambda^2T^*M)$ that is non-degenerate, i.e.,  for every chart $(W,\psi)$ and $\tilde{z} \in {\psi(W)}^{\sim}_c$ the $\Rt$-bilinear form $ \psi_* \sig (\tilde{z}) \col \Rt^d \times \Rt^d \to \Rt$  is non-degenerate .
\ed
 
As was already mentioned in the introduction, the fact that $ \psi_* \sig$ is a symplectic form
on $\Rt^d$ implies that $d$ is even, say $d=2n$.

Given any $(0,2)$-tensor field $\alpha$ on $M$ and a coordinate system $(\psi=(x^1,\dots,x^d),U)$ we 
may write $\alpha|_{U} = \alpha_{ij} dx^i\otimes dx^j$ and set $\vol(\alpha)|_U:=\sqrt{|\det (\alpha_{ij})|}$.
As this quantity transforms by multiplication with the Jacobian determinant of the chart transition functions,
we obtain a well-defined $1$-density $\vol(\alpha)$ on $M$. In the case where $\alpha$ is a Riemannian metric
on $M$, $\vol(\alpha)$ is the Riemannian volume density of $\alpha$. A component-wise application of the above
procedure to a generalized $(0,2)$-tensor %(or, in particular, to a generalized $2$-form) 
$\alpha$ yields
a corresponding generalized one-density $\vol(\alpha)\in \Gamma_\cG(\Vol(M))$ (with $\Vol(M)$ the $1$-density bundle 
over $M$).

\bprop \label{skewchar}
Let $\sigma\in\cG_2^0(M)$. Then the following are equivalent:
\begin{itemize}
\item[(i)] $\sigma:\mathfrak{X}(M)\times \mathfrak{X}(M)\to \cG(M)$
is skew-symmetric $($equivalently, $\sigma\in \Gamma_\cG(\Lambda^2T^*M))$ and $\vol(\sigma)$ is strictly positive.
\item[(ii)] For each chart $(W,\psi)$ and for each $\tilde{z}\in\psi(W)^\sim_c$, the map
$\psi_*\sigma(\tilde{z}):\Rt^{d}\times\Rt^{d} \to\Rt$ is skew-symmetric and non-degenerate.
\item[(iii)] $\vol(\sigma)$ is strictly positive and for each relatively compact open set $V\subseteq M$ there
exist a representative $(\sigma_\eps)_\eps$ of $\sigma$ and $\eps_0>0$ such that $\sigma_\eps|_V$ is skew-symmetric
%and non-degenerate 
for all $\eps<\eps_0$. 
\end{itemize}
\eprop

\pr
(i) $\Rightarrow$ (ii)
By the point-value characterization of invertibility of generalized functions (\cite[1.2.55]{book}),
strict positivity of $\vol(\sigma)$ implies that $\det(\psi_*\sigma(\tilde z))$ is invertible in $\Rt$
for each $\tilde z\in M^\sim_c$. Hence by \cite[1.2.41]{book}, $\psi_*\sigma(\tilde z): 
\Rt^{d}\times\Rt^{d} \to\Rt$ is non-degenerate. Skew-symmetry of
$\psi_*\sigma(\tilde{z})$ follows by inserting basis vector fields in a chart and evaluating.

(ii) $\Rightarrow$ (i) Given $X_1,X_2\in \frak{X}(M)$ it follows from (ii)
and the point value characterization of Colombeau functions
that $\sigma(X_1,X_2)|_W=-\sigma(X_2,X_1)|_W$ on any chart domain $W$. 
This gives skew-symmetry since $\cG$ is a sheaf.
By \cite[1.2.41]{book}, for any $\tilde z\in\psi(W)^\sim_c$,
$\det(\psi_*\sigma(\tilde z))$ is invertible, hence ($\psi_*\sigma(\tilde z)$ being skew-symmetric) 
strictly positive, in $\Rt$. Therefore $\vol(\sigma)|_W$ is strictly positive. Since $W$ was any chart domain,
$\vol(\sigma)$ is strictly positive on $M$. 

(i) $\Rightarrow$ (iii) Using a partition of unity the problem can be reduced to the
case of $M=\R^d$. Now pick any representative $(\sigma_\eps)_\eps$ of $\sigma$. Then
$(\tilde\sigma_\eps)_{ij}:= \frac{1}{2}((\sigma_\eps)_{ij}-(\sigma_\eps)_{ji})$ gives
a skew-symmetric representative of $\sigma$ (cf.\ \cite[Lemma 3.23]{HKK:14}). 

(iii) $\Rightarrow$ (i)
Let $X_1$, $X_2\in \frak{X}(M)$ and let $V$ be any relatively compact open subset
of $M$. Picking a representative as in (iii) it is clear that $\sigma(X_1,X_2)=-\sigma(X_2,X_1)$ 
on $V$. Then the sheaf property of $\cG$ gives skew-symmetry.
\ep

To obtain a characterization of symplectic generalized forms from this result, we will use the following
generalized Poincar\'e lemma (see \cite[3.2.40]{book}):

\bthm\label{poincare} Let $\alpha\in \Gamma_\cG(\Lambda^kT^*M)$ be closed. If $p\in M$ and $U$ is a neighborhood of $p$
that is diffeomorphic to an open ball in $\R^d$ then there exists $\beta\in \Gamma_\cG(\Lambda^{k-1}T^*M)$
such that $\alpha|_U=d\beta|_U$.
\ethm

\bcor\label{sympchar} Let $\sigma\in\cG_2^0(M)$. Then the following are equivalent:
\begin{itemize}
\item[(i)]  $\sigma$ is a generalized symplectic form on $M$.
\item[(ii)] $\vol(\sigma)$ is strictly positive and
for every $p\in M$ there exists an open neighborhood $U$ of $p$ and a representative
$(\sigma_\eps)_\eps$ of $\sigma|_U$ such that each $\sigma_\eps$ is a symplectic form on $U$. 
\end{itemize}
\ecor

\pr (ii) $\Rightarrow$ (i) is immediate from Prop.\ \ref{skewchar} (iii).

(i) $\Rightarrow$ (ii) Given $p\in M$, let $U$ be a relatively compact neighborhood of $p$ 
diffeomorphic (via some chart $\psi$) to an open ball in $\R^d$. Then Prop.\ \ref{skewchar} (iii)
provides a representative $(\tilde \sigma_\eps)_\eps$ of $\sigma$ as well as some 
$\eps_0>0$ and $m\in \N$ such that $\det (\psi_*\tilde\sigma_\eps(x))
>\eps^m$ for $\eps<\eps_0$ and all $x\in \psi(U)$, so $\tilde \sigma_\eps|_U$ is 
a non-degenerate $2$-form. Since $\sigma$ is closed, there exists 
a negligible net $(n_\eps)_\eps$ of $3$-forms such that for all $\eps$ we have $d\tilde\sigma_\eps= n_\eps$.
Thus $dn_\eps=0$ for all $\eps$. By the classical Poincar\'e lemma we may write $n_\eps=dm_\eps$ on $U$,
and (the proof of) Th.\ \ref{poincare} shows that $(m_\eps)_\eps$ is a negligible net of $2$-forms on $U$. 
Thus $\sigma_\eps:=\tilde \sigma_\eps - m_\eps$ is a representative of $\sigma|_U$ with 
$d\sigma_\eps=0$ for all $\eps$. Finally, non-degeneracy of $\sigma_\eps$ for $\eps$ small follows since
$\vol(\sigma)$ is strictly positive by Prop.\ \ref{skewchar}.
\ep

Note that $\sig_\eps$ is non-degenerate if and only if the $n$-fold exterior product 
$\sig_\eps^n := \sig_\eps \wedge \sig_\eps \wedge \ldots \wedge \sig_\eps$ provides a 
volume form on $M$ (with $\dim(M)=d=2n$).

%%%%%%%%%%%%%%%%%%%%%%%%%%%%%%%%%%%%%%%%%%%%%%%%%%%%%%%%%%%%%%%%%%
 \section{A generalized Darboux theorem for non-smooth symplectic forms} \label{sec:Darboux}
%%%%%%%%%%%%%%%%%%%%%%%%%%%%%%%%%%%%%%%%%%%%%%%%%%%%%%%%%%%%%%%%%%

In his fundamental work on the distributional approach to non-smooth mechanics \cite{marsden},
J.\ E.\ Marsden states: ``{\em It is meaningful to talk about generalized
symplectic forms although this does not lead to a satisfactory theory. Clearly Darboux's theorem 
cannot hold in that case.}''

The problem we address in this section is: let $\sig$ be a generalized symplectic form on the manifold $M$
(in the sense described in the previous section).  Can we find {\em generalized} Darboux coordinates? That is, for any $p\in M$ we seek an open neighborhood $U$ of $p$ and a generalized diffeomorphism  $\Phi  \in \cG[U,\R^{2n}]$ such that 
$$
  \sig = \Phi^* \widetilde{\om}, 
$$
where $\widetilde{\om}$ is the canonical $2$-form on $T^*(\widetilde{\R}^n)$.

Weinstein's proof  (based on an isotopy method by Moser from \cite{Moser:65})  outlines the following basic steps to construct Darboux coordinates (cf.\ \cite{Weinstein:71}; see also \cite[Section 3.2]{AM:78}, \cite[Section 3.2]{McDS:98}, or \cite[Chapter 22]{Lee:13}) --- we add remarks on the key issues in extending this to the non-smooth situation below:
\renewcommand{\labelenumi}{\arabic{enumi}.}
\begin{enumerate}

\item We are dealing with a local question, hence  we may assume without loss of generality that $M= \R^{2n}$ and $p=0$.

\item Define the  constant symplectic form $\be$ on $\R^{2n}$ by $\be{(x)} := \sig{(0)}$ for every $x \in \R^{2n}$ and put
$$
    \mu_t := \sig + t \, (\be - \sig) \quad (0 \leq t \leq 1).
$$
Note that $d \mu_t = 0$, since $\sig$ is closed and $\be$ is constant; moreover, $\mu_0
= \sig$ and $\mu_1 = \be$.

\item On some open ball around $0$, say $B_R(0)$, we have that $\mu_t$ is non-degenerate for every $t \in [0,1]$. In the case of generalized $2$-forms we will need a condition, called condition  $(\star)$ below, ensuring that the same holds uniformly for small values of the  parameter $\eps$, if $(\mu^\eps_t)_{\eps \in \,]0,1]}$ is a family representing $\mu_t$.

\item Applying Poincar\'{e}'s lemma on $B_R(0)$ there is a $1$-form $\al$ such that $d \al = \be - \sig$  and $\al(0) = 0$.

\item Define a vector field $X_t$ on $B_R(0)$ by requiring  $\mu_t(X_t,.) = - \al$. Since this corresponds to a pointwise inversion of the matrices representing $\mu_t$,  the non-smooth analogue will depend on the precise information from the ``non-degeneracy'' condition $(\star)$ in a crucial way.

\item Since ${X_t}(0) = 0$ we have an evolution $\theta_{t,s}$ for the time-dependent vector field $X_t$ defined up to $t=1$ on some possibly smaller neighborhood $U$ of $0$ (note that $\theta_{t,0}(0) = 0$ for all $t$ and $\theta_{0,0}$ is the identity). In the non-smooth case, again exploiting details from condition ($\star$) will be essential to establish a generalized evolution correspondingly.

\item The  result of the following calculation  (using \cite[Proposition 22.15]{Lee:13} to obtain the first equality)
\begin{multline*}
  \diff{t} (\theta_{t,0}^*\, \mu_t) = 
  \theta_{t,0}^*(L_{X_t}\, \mu_t) + \theta_{t,0}^*\, \diff{t} \mu_t\\ 
  =   \theta_{t,0}^*(d\, i_{X_t}\, \mu_t + i_{X_t}\, d \mu_t) +  \theta_{t,0}^*(\be - \sig) = \theta_{t,0}^*(- d \al + \be - \sig) = 0
\end{multline*}
 implies  $\theta_{1,0}^*\, \be = \theta_{1,0}^*\, \mu_1 = 
\theta_{0,0}^*\, \mu_0 = \sig$, hence in coordinates corresponding to the diffeomorphism $\theta_{1,0}$ the symplectic form $\sig$ is the constant form $\be$.

\item Finally, we may map $\be$ into the canonical $2$-form (or $\widetilde{\om}$ in the generalized setting) by the Darboux-analogue of symplectic linear algebra (\cite[Th.\ 3.3]{HKK:14}), which in combination with the previous step yields the desired transformation of $\sig$.
\end{enumerate}
\renewcommand{\labelenumi}{(\roman{enumi})}

Turning now to the generalized setting, let $\sig$ be a generalized symplectic form 
on the smooth $2n$-dimensional manifold $M$. The basic condition on $\sig$, which guarantees the existence of generalized Darboux coordinates, is that $\sigma$ should possess a representative $(\sig_\eps)_\eps$ satisfying:

\begin{description}
\item{($\star$)}  The family $(\sig_\eps)_{\eps \in\, ]0,1]}$ of maps $M \to \Lambda^2 T^*M$ 
is {equicontinuous}  and satisfies the following on any chart $(W,\psi)$ with domain $W \subseteq M$ and matrix representation $\Omega^\eps \col W \to M(2n, \R)$ of $\sig_\eps$ with respect to this chart:  $\forall K \Subset W\; \exists C_1, C_2 > 0\; \exists \eta > 0\; \forall \eps \in\, ]0,\eta]$ $\;\forall q \in K$
$$
  0 < C_1 \leq \min  \cA(\Omega^\eps(q)) \leq  \max \cA(\Omega^\eps(q)) \leq C_2.
$$
\end{description}
Here, for a matrix $B \in M(2n,\R)$ we let $\cA(B) := \{ |\la| \mid \la \text{ is an eigenvalue of } B \}$.

\brem \ 
\begin{itemize}
\item[(i)] Condition ($\star$) holds for the typical convolution-type regularization of a uniformly continuous symplectic form. 
\item[(ii)] ($\star$) requires that the $\eps$-parametrized family of volume forms $\sig_\eps^n := \sig_\eps \wedge \sig_\eps \wedge \ldots \wedge \sig_\eps$ has uniform bounds on how a $2n$-dimensional volume is squeezed or stretched.
\item[(iii)]  
If $\sigma$ satisfies ($\star$) and, in addition, is associated to some $\sigma_0\in \Omega^2_{\mathcal D'}(M)$,
then by Arzela-Ascoli there exists some sub-sequence $(\sigma_{\eps_k})_{k\in \N}$ that converges locally uniformly
to a continuous two-form $\sigma_1$. Since $\sigma_\eps \to \sigma_0$ in distributions it follows that $\sigma_0=\sigma_1$
is continuous.  
\item[(iv)] The generalized $2$-forms $\eps dx\wedge d\xi$ and $\frac{1}{\eps} dx\wedge d\xi$  are simple examples
of generalized symplectic forms that do not satisfy ($\star$), yet can be transformed to the canonical symplectic form.
However, no such transformation can be c-bounded.
\item[(v)] Writing $H$ and $y_+$ for embeddings of the Heaviside function resp.\ its primitive into the Colombeau algebra,
consider $\sigma:= (1+H(x))dx\wedge d\xi$. Then (by (iii)), $\sigma$ does not satisfy ($\star$). Nevertheless,
setting $\Psi(y,\eta):=(y-\frac{1}{2}y_+,\eta)$ we obtain $\Psi_*\sigma = dx\wedge d\xi$. Moreover, $\Psi$
is a generalized diffeomorphism, and both $\Psi$ and $\Psi^{-1}$ are associated to Lipschitz continuous transformations.
\end{itemize}
\erem

\bthm Let $\sig$ be a generalized symplectic form on the $2n$-dimensional smooth manifold $M$ with representative $(\sig_\eps)_{\eps \in \,]0,1]}$ such that $(\star)$ holds. Then every $p \in M$ possesses a neighborhood $U$ 
and a generalized  
diffeomorphism $\Phi \in \cG[U,V]$ with $V$ open in $T^*(\R^n)$ such that $\Phi_* (\sig \!\mid_U) = \widetilde{\om}$.  
\ethm

\pr  \emph{Step 1:} This can be carried out as in the above scheme:
we may assume that $M=\R^{2n}$ and that $p=0$. Then each $\sigma_\eps$ is a skew-symmetric and 
non-degenerate matrix.  
In particular, $\Omega^\eps=\sigma_\eps$ in ($\star$).  

\emph{Step 2:}
Setting $\beta:=\sigma(0)$, $\be$ is a symplectic form on the $\Rt$-module 
$T^*({\Rt}^n) = \Rt^n \times \Rt^n$ and clearly $d\beta=0$.
For $t\in[0,1]$, set $\mu_t^\eps := \sig_\eps + t \, (\be_\eps - \sig_\eps)$.
Then $\mu_0 = \sig$ and $\mu_1 = \be$ in  
$\Gamma_\cG(\Lambda^2T^*\R^{n}) = \Omega_\cG(\R^{2n})$, and $d \mu_t = 0$.  

\emph{Step 3:} By assumption ($\star$), the family of matrix-valued maps
$$
\mu^\eps:=(t,q) \mapsto \mu_t^\eps(q), \quad [0,1] \times \R^{2n} \to M(2n, \R) \qquad (\eps\in ]0,1])
$$
is equicontinuous. We show that for small $\eps > 0$, $\mu^\eps$ maps $[0,1]$ times some fixed  neighborhood of $0$ into the set of invertible matrices (corresponding to the non-degeneracy of the associated $2$-forms) with uniform bounds on the operator norms of the matrices and their inverses. The precise claim is as follows:

\begin{description}
\item{($\star\star$)} $\exists R > 0$ $\exists \eps_0 > 0$ $\exists D > 0$ such that $\forall \eps \in \, ]0,\eps_0]$ $\forall t \in [0,1]$ $\forall q \in B_R(0)$ the matrix $\mu_t^\eps(q)$ is invertible and
$$ 
       \opnorm{\mu_t^\eps(q)} \leq D, \quad \opnorm{\mu_t^\eps(q)^{-1}} \leq D.
$$
\end{description}

To see this, note first that uniform boundedness of $\mu_t^\eps(q)$ for $q \in B_1(0)$ (or any relatively compact subset of $\R^{2n}$) and small $\eps$ follows from 
$$
  \opnorm{\mu_t^\eps(q)} \leq 
    (1-t) \opnorm{\sigma_\eps(q)} + t \,\opnorm{\sigma_\eps(0)} 
    \leq \opnorm{\sigma_\eps(q)} + \opnorm{\sigma_\eps(0)},
$$
recalling that the operator norm equals the spectral radius for (skew-symmetric hence) normal operators, and finally applying ($\star$) with $K = \overline{B_1(0)}$. Hence there is some $\eta \in \,]0,1]$ and $C_2 > 0$ such that
$$
    \opnorm{\mu_t^\eps(q)} \leq 2\, C_2 \qquad \forall \eps \in \,]0,\eta]\; \forall t \in [0,1]\;
       \forall q \in B_1(0).
$$
We will also establish invertibility and uniform boundedness of the family of inverses on some ball $B_R(0)$ with $0< R \leq 1$ and for $0 < \eps \leq \eta$. To this end, we note that $- \sigma_\eps(0)^2$ is a (self-adjoint) positive-definite operator, hence 
$$
   \inf\limits_{\norm{v}=1} \norm{\sigma_\eps(0) v}^2  = 
    \inf\limits_{\norm{v}=1} v^T \cdot (- \sigma_\eps(0)^2) \cdot v =
    \min \cA(-\sigma_\eps(0)^2) = \big( \min \cA(\sigma_\eps(0)) \big)^2.
$$ 
Let $C_1$ denote the (positive) lower bound in ($\star$) applied as above with $K = \overline{B_1(0)}$, then the previous observation gives 
$\norm{\sigma_\eps(0) v} \geq \min \cA(\sigma_\eps(0)) \norm{v} \geq C_1 \norm{v}$ and therefore 
$$
     \opnorm{\sigma_\eps(0)^{-1}} \leq \frac{1}{C_1} \quad  \forall \eps \in \,]0,\eta].
$$

In completing the proof of $(\star\star$) we will make use of the following well-known fact about invertibility in normed algebras with unit (e.g., \cite[p.\ 177]{KR:97}): 
If $A$ is invertible and $\norm{B - A} < 1/\norm{A^{-1}}$, then $B$ is invertible and $\norm{B^{-1} - A^{-1}} \leq \norm{B - A} \norm{A^{-1}}^2/(1 - \norm{A^{-1}} \norm{B - A})$. We will apply this to the situation with $A = \sigma_\eps(0)$ and $B = \mu_t^\eps(q)$. 

Since we have shown $C_1 \leq 1/\opnorm{\sigma_\eps(0)^{-1}}$ above, proving a uniform estimate of the form 
\beq \tag{$\Delta$}
  \opnorm{\mu_t^\eps(q) - \sigma_\eps(0)} \leq C_1/2 \qquad 
     \forall \eps \in \, ]0,\eta] \;\forall t \in [0,1] \; \forall q \in B_R(0)
\eeq
will suffice to establish invertibility of $\mu_t^\eps(q)$ and a uniform bound 
$$
   \opnorm{\mu_t^\eps(q)^{-1}} \leq  
   \opnorm{\mu_t^\eps(q)^{-1} - \sigma_\eps(0)^{-1}} + 
      \opnorm{\sigma_\eps(0)^{-1}} \leq 
      \frac{\frac{C_1}{2} \frac{1}{C_1^2}}{1 - \frac{1}{C_1} \frac{C_1}{2}} + 
          \frac{1}{C_1} = \frac{2}{C_1}.
$$
To argue that ($\Delta$) holds for some $0 < R \leq 1$ we simply call on the equicontinuity of 
$(\sigma_\eps)_{\eps\in\,]0,1]}$ to establish the last inequality in the following chain:
$$
   \opnorm{\mu_t^\eps(q) - \sigma_\eps(0)} = 
     \opnorm{(1-t) \big(\sigma_\eps(q) - \sigma_\eps(0)\big)} \leq
     \opnorm{\sigma_\eps(q) - \sigma_\eps(0)} \leq \frac{C_1}{2}.
$$
Therefore ($\star\star$) holds with $D := \max(2 C_2, 2/C_1)$.

\emph{Step 4:}
By Th.\ \ref{poincare} we may construct a generalized $1$-form $\alpha$ on $B_R(0)$ such that
$d\alpha = \beta - \sigma$. Moreover, the proof of \cite[3.2.40]{book} shows that
$\alpha$ has a representative given by
\begin{equation}\label{Hdef}
\alpha_\eps(q)(v) = \int_0^1 t(\beta_\eps - \sigma_\eps)(tq)(q,v)\,dt \qquad (v\in \R^{2n})
\end{equation}
In particular, $\alpha_\eps(0)=0$ for all $\eps$.

\emph{Step 5:}
By non-degeneracy of $\mu_t^\eps$, for each $\eps$ there exists a unique vector field $X^\eps_t$
on $B_R(0)$ such that $\mu_t^\eps(X_t^\eps,\,.\,)=-\alpha_\eps$ on $B_R(0)$ for all 
$t\in [0,1]$ and
all $\eps$. Moderateness of $\alpha_\eps$ and $\mu_t^\eps$, together with boundedness of
$(\mu_t^\eps)^{-1}$ imply that the net $(X^\eps_t)_\eps$ defines a generalized time-dependent vector field on
$B_R(0)$, satisfying global bounds with respect to $t\in [0,1]$. Next, let $R'< R$ and
pick a smooth, compactly supported plateau function $\vphi: \R^{2n}\to \R$ such that
$\vphi\equiv 1$ on a neighborhood of $B_{R'}(0)$. Then $Y^\eps_t(q):=\vphi(q)X^\eps_t(q)$
defines a global time-dependent generalized vector field $Y_t$ that coincides with $X_t$ on
$B_{R'}(0)$.

\emph{Step 6:} 
Denote by $\theta^\eps$ the flow of $Y^\eps_t$, i.e.,
\begin{align*}
\frac{d}{dt}\theta^\eps(t,s,q) &= Y^\eps_t(\theta^\eps(t,s,q))\\
\theta^\eps(s,s,q) &= q
\end{align*}
Since by $(\star)$ and the above construction $Y^\eps_t$ is globally bounded, uniformly
in $t$ and $\eps$, each $\theta^\eps$ is defined on all of $[0,1]\times [0,1]\times \R^{2n}$.
Thus we obtain a c-bounded generalized function $\theta=[(\theta^\eps)_\eps]\in 
\cG[[0,1]\times [0,1]\times \R^{2n},\R^{2n}]$, and for each fixed $(t,s)\in [0,1]^2$
an invertible generalized map $\theta_{t,s} := \theta(t,s,\,.\,)$.

By condition ($\star$) it follows that the family $X^\eps_t$ ($\eps\in ]0,1]$) is equicontinuous,
uniformly in $t$. The same therefore is true of $Y^\eps_t$ ($\eps\in ]0,1]$). Together with the
global boundedness of $Y^\eps_t$ and the continuous dependence of $\theta^\eps$ on the right hand
side of its defining equation this implies that also $\theta^\eps$ ($\eps\in ]0,1]$) is 
equicontinuous. Since $Y^\eps_t(0)=0$, and thereby $\theta^\eps(t,0,0) = 0$ for all $t$,
it follows that there exists some 
$R''<R'$ such that $\theta^\eps(t,0,q)\in B_{R'}(0)$ for all $\eps$, all $q\in B_{R''}(0)$ 
and all $t\in [0,1]$.

\emph{Step 7:} 
By the above preparations, on $B_{R''}(0)$ we can calculate as follows: 
\begin{multline*}
  \diff{t} ((\theta^\eps_{t,0})^*\, \mu^\eps_t) = 
  (\theta^\eps_{t,0})^*(L_{Y^\eps_t}\, \mu^\eps_t) + (\theta^\eps_{t,0})^*\, \diff{t} \mu^\eps_t =
(\theta^\eps_{t,0})^*(L_{X^\eps_t}\, \mu^\eps_t) + (\theta^\eps_{t,0})^*\, \diff{t} \mu^\eps_t
\\ 
  =   (\theta^\eps_{t,0})^*(d\, i_{X^\eps_t}\, \mu^\eps_t + i_{X^\eps_t}\, d \mu^\eps_t) +  (\theta^\eps_{t,0})^*(\be_\eps - \sig_\eps) = 
  (\theta^\eps_{t,0})^*(- d \al_\eps + \be_\eps - \sig_\eps).
\end{multline*}
Integrating, we obtain 
$$
(\theta^\eps_{1,0})^*\beta_\eps = (\theta^\eps_{1,0})^*\, \mu^\eps_1 = (\theta^\eps_{0,0})^*\mu^\eps_0 + \nu_\eps
= \sigma_\eps + \nu_\eps
$$
with $(\nu_\eps)_\eps$ a negligible $2$-form. Thus $\sigma = (\theta^\eps_{1,0})^*\beta$ as a generalized $2$-form.

Consequently, the generalized diffeomorphism $\theta_{1,0}$ transforms $\sigma$ into the constant symplectic form $\beta$. 

\emph{Step 8:} 
Finally, by choosing (according to \cite[Th.\ 3.3]{HKK:14}) a symplectic basis on ${\Rt}^{2n}$ 
corresponding to $\beta$  we transform
$\beta$ to the canonical symplectic form $\widetilde \omega$. 
\ep

%%%%%%%%%%%%%%%%%%%%%%%%%%%%%%%%%%%%%%%%%%%%%%%%%%%%%%%%%%%%%%%%%%
 \section{Generalized Hamiltonian vector fields and Poisson structures} \label{sec:Poisson}
%%%%%%%%%%%%%%%%%%%%%%%%%%%%%%%%%%%%%%%%%%%%%%%%%%%%%%%%%%%%%%%%%%
The natural next step in the development of generalized symplectic geometry
is the introduction of Hamiltonian vector fields and, building on this,
Poisson structures. This final section is devoted to providing these 
notions.

To begin with, we analyze the purely algebraic setup, based on \cite{HKK:14}. 
Given an $\Rt$-module $V$, its dual $\Rt$-module $\mathrm{L}(V,\Rt)$ 
is denoted by $V'$.

\begin{lem}
Let $(V,\sigma)$ be a symplectic $\Rt$-module that is free and of finite rank. 
Then for any $\vphi\in V'$ there exists a unique $h_\vphi\in V$, the
Hamiltonian vector corresponding to $\vphi$, such that
$$
\forall v\in V: \quad \vphi(v) = \sigma(h_\vphi,v).
$$
Furthermore, the map $\vphi \mapsto h_\vphi$, $V'\to V$ is a linear isomorphism
\end{lem}
\begin{proof}
Uniqueness is immediate by the non-degeneracy of $\sigma$. To prove existence,
picking any basis of $V$ and the corresponding dual basis on $V'$ we can rewrite the
defining equation for $h_\vphi$ in matrix form as 
$\vphi\cdot v = h_\vphi^T \cdot \sigma\cdot v$. Hence 
\begin{equation}
\label{hamilton}
h_\vphi^T = \vphi\cdot \sigma^{-1}.
\end{equation}
More explicitly, in terms of a symplectic basis $(e_1,\dots,e_n,f_1,\dots,f_n)$ of $V$, 
we have
$$
h_\vphi = \sum_{i=1}^n (\vphi(f_i)e_i - \vphi(e_i)f_i).
$$
The final claim is immediate from the construction.
\end{proof}
Based on this result, we can introduce the Poisson bracket 
\begin{align*}
 \{\,\,,\,\}:V'\times V' &\to \Rt \\
 \{\vphi,\psi\} &= \sigma(h_\vphi,h_\psi)
\end{align*}
As in the vector space setting it is easily seen that $\{\,\,,\,\}$ is skew-symmetric,
non-degenerate and satisfies the Jacobi-identity. Thus we obtain:
\begin{prop}
Let $(V,\sigma)$ be a symplectic $\Rt$-module that is free and of finite rank.
Then $(V',\{\,\,,\,\})$ is a symplectic $\Rt$-module that is free and of the 
same rank as $V$.
\end{prop}
Turning now to the manifold setting, we have:
\begin{thm}
Let $(M,\sigma)$ be a generalized symplectic manifold. Then the mapping
\begin{align*}
 \sigma_\flat: \mathfrak{X}_\cG(M) &\to \Omega^1_\cG(M) \\
 \sigma_\flat(X)(Y) &:= \sigma(X,Y) \quad (X,Y \in \mathfrak{X}_\cG(M))
\end{align*}
is a $\cG(M)$-linear isomorphism. Its inverse will be denoted by $\sigma^\sharp$.
\end{thm}
\begin{proof}
Since $Y\mapsto \sigma(X,Y)$ is $\cG(M)$-linear, $\sigma_\flat(X)\in \Omega^1_\cG(M)$ for
any $X\in \mathfrak{X}_\cG(M)$ (see \cite[3.2.27]{book}), and clearly $\sigma_\flat$
is $\cG(M)$-linear. To prove injectivity, suppose that $\sigma(X,Y)=0$ for all 
$Y\in \mathfrak{X}_\cG(M)$. Then for any chart $(W,\psi)$ it follows that 
$\psi_*(\sigma)(\tilde z)(\psi_*(X)(\tilde z),w)=0$ for any $\tilde z\in \psi(W)^\sim_c$
and any $w\in \Rt^{2n}$. Thus non-degeneracy gives $\psi_*(X)(\tilde z)=0$,
implying that $X=0$.

To show surjectivity, by the sheaf property it suffices to consider the case $M=\R^{2n}$.
Let $\alpha\in \Omega^1_\cG(\R^{2n})$. Then, using \eqref{hamilton}, in matrix notation
we may set $X_\alpha^T := \alpha\cdot \sigma^{-1}$. By the positivity of $\det(\sigma)$
in $\cG(\R^{2n})$ (Cor.\ \ref{sympchar}), this
defines a generalized vector field $X_\alpha\in \mathfrak{X}_\cG(M)$, and by construction
$\sigma_\flat(X_\alpha)=\alpha$.
\end{proof}
\begin{rem} Note that the previous result did not make use of generalized Darboux
coordinates (which in general may not be available), but is valid for arbitrary
generalized symplectic manifolds.
\end{rem}

Based on this result, we may now define:
\begin{defin} 
Let $(M,\sigma)$ be a generalized symplectic manifold. For any $f\in \cG(M)$,
the generalized vector field $H_f:=\sigma^\sharp(df)\in \mathfrak{X}_\cG(M)$
is called the Hamiltonian vector field of $f$. 

Moreover, for $f$, $g\in \cG(M)$, the Poisson bracket of $f$ and $g$ is
given by $\{f,g\}:=\sigma(H_f,H_g)$.
\end{defin}

Analogous to the smooth setting, the Poisson bracket induces an $\Rt$-Lie-algebra
structure on $\cG(M)$.

Summing up, the above constructions provide the foundations for a regularization-based
approach to non-smooth symplectic geometry. Contrary to the distributional setting,
the additional flexibility of Colombeau's theory allows one to retain the basic
structure of the smooth setting. Building on these constructions one may now
systematically explore applications, in particular in non-smooth mechanics.
In particular, previous work in this direction (e.g., \cite{HO:99,KKO:08}) can now be viewed from a
unifying perspective. Finally, we hope that this approach will be useful in studying the
propagation of singularities for pseudo-differential
operators with non-smooth principal symbol on differentiable manifolds.

%%%%%%%%%%%%%%%%%%%%%%%%%%%%%%%%%%%%%%%%%%%%%%%%%%%%%%%%%%%%%%%%%%
 \subsection*{Acknowledgment}
%%%%%%%%%%%%%%%%%%%%%%%%%%%%%%%%%%%%%%%%%%%%%%%%%%%%%%%%%%%%%%%%%%

This work was supported by project P25236 of the
Austrian Science Fund and  projects 174024 
of the Serbian Ministry of Science, and
114-451-3605 of the Provincial Secretariat for Science.

%%%%%%%%%%%%%%%%%%%%%%%%%%%%%%%%%%%%%%%%%%%%%%%%%%%%%%%%%%%%%%%%%%
%%%%%%%%%%%%%%%%%%%%%%%%%%%%%%%%%%%%%%%%%%%%%%%%%%%%%%%%%%%%%%%%%%

%%%%%%%%%%%%%%%%%%%%%%%%%%%%%%%%%%%%%%%%%%%%%%%%%%%%%%%%%%%%%%%%%%
%%%%%%%%%%%%%%%%%%%%%%%%%%%%%%%%%%%%%%%%%%%%%%%%%%%%%%%%%%%%%%%%%%
 \end{document}